\documentclass{amsart}%
\usepackage{amssymb,mathabx}
\usepackage{color}
\usepackage{amsmath}
\usepackage{graphicx}
\usepackage{amsfonts}%
\usepackage[colorlinks=false]{hyperref}
\usepackage{setspace}
\usepackage{etoolbox}
\usepackage{tikz-cd}
\usetikzlibrary{matrix,arrows,decorations.pathmorphing,decorations.markings} 
\usepackage{adjustbox}
\usepackage[pagewise]{lineno}
\usepackage{multicol}
\usepackage{tikz}
\usepackage{microtype}

\newcommand{\F}{\mathbb F}

\usepackage{bm}
\usepackage{amsfonts, stmaryrd}
\usepackage[bbgreekl]{mathbbol}
\DeclareSymbolFontAlphabet{\mathbbm}{bbold}
\DeclareSymbolFontAlphabet{\mathbb}{AMSb}
\usepackage{mathrsfs}
\usepackage{tikz-cd}
\usetikzlibrary{matrix,arrows,decorations.pathmorphing,decorations.markings} 
\usepackage{pb-diagram}
\usepackage{enumerate}

\usepackage{hyperref}

\newtheorem{theorem}{Theorem}[section]

\newtheorem{corollary}[theorem]{Corollary}

\newtheorem{definition}[theorem]{Definition}
\newtheorem{example}[theorem]{Example}

\newtheorem{proposition}[theorem]{Proposition}
\newtheorem{remark}[theorem]{Remark}

\numberwithin{equation}{section}
\AtBeginEnvironment{equation}{\leavevmode\singlespace}
\AfterEndEnvironment{equation}{\endsinglespace\vskip0.5\baselineskip\noindent\ignorespaces}

\newcommand{\bk}{\Bbbk}
\usepackage{twoopt}

\newcommandtwoopt\eck[3][G][\Bbbk]{H^*_{#1}(#3; #2)}

\newcommandtwoopt\csk[2][G][\Bbbk]{H^*(B{#1}; #2)}

\DeclareMathOperator{\tor}{Tor} %tor Functor
\DeclareMathOperator{\coker}{Coker}

\DeclareMathOperator{\im}{im}

\title[Equivariant cohomology for cyclic group actions]{
	The equivariant cohomology for cyclic group actions on some polyhedral products
}
\date{ \today }

\author[Chaves]{Sergio Chaves}
\address[Sergio Chaves]{Department of Mathematics\\ University of Rochester}
\email{schavesr@math.rochester.edu}

\usepackage{tikz-cd}
\usetikzlibrary{matrix,arrows,decorations.pathmorphing,decorations.markings} 

\usetikzlibrary{hobby}
\usetikzlibrary{scopes,intersections}
\usepackage{adjustbox}
\usepackage{multicol}

\usepackage{pgfplots}
\pgfplotsset{compat=1.9}
\newcommand{\Coordinate}[2]%
{ \coordinate (#1) at (#2);
}

\usepackage{spectralsequences}

\begin{document}

\begin{abstract}
The equivariant cohomology for actions of compact connected abelian groups and elementary abelian $p$-groups have been widely studied in the last decades. We study some of these results on actions of finite cyclic groups over a field of positive characteristic. In particular, we study the module structure of the equivariant cohomology over the cohomology of the classifying space and we provide a criterion for freeness over this ring and the polynomial subring of it. We apply these results to canonical actions of finite cyclic groups on polyhedral products arising from the boundary of a polygon.
\end{abstract}

\maketitle

\section{Introduction}\label{se:intro}
Let $G$ be a topological group. For any $G$-space $X$, the  $G$-equivariant cohomology of $X$ with coefficients over a field $\bk$ is defined as the cohomology of the homotopy quotient  $H^*_G(X; \bk) := H^*( EG \times_G X; \bk)$. The equivariant cohomology inherits a canonical module structure over the ring $R_G = H^*(BG;\bk)$. We say that $X$ is  $G$-equivariantly formal over $\bk$ if the restriction map $H^*_G(X;\bk) \rightarrow H^*(X;\bk)$ is surjective. This is equivalent to $H^*_G(X;\bk)$ being a free $R_G$-module and that $H^*(X;\bk)$ is a trivial $G$-module. This is a consequence of the Leray-Hisch Theorem and a particular converse for it \cite{ptori}[Prop.4.4]

The notion of equivariant formality has played an important role in the theory of toric varieties, toric manifolds, symplectic manifolds and polyhedral products, where geometric and combinatorial properties of the space and the action give rise to characterization of equivariant formality  \cite{panov}, \cite{januz}, \cite{lupanov}. In the last decades, this has been a topic of interest for compact connected Lie group actions (and torus actions with an extended detail) \cite{Mab}, \cite{Mtori}  elementary abelian group actions \cite{ptori} and some semidirect product actions \cite{semi}. In this document, we will study the equivariant cohomology for cyclic group actions when cohomology with coefficients over a field of positive characteristic is considered.  

If $G$ is a finite group, the cohomology of the classifying space $H^*(BG;\bk)$ is canonically identified with the group cohomology $H^*_{grp}(G;\bk)$. We will use the notation $H^*(BG;\bk)$ or $H^*_{grp}(G;\bk)$ indistinguishably to refer to the ring $R_G$. In this document, we focus on the equivariant cohomology for actions of finite cyclic groups. We also refer by a $G$-space $X$ to a finite $G$-CW complex so the equivariant cohomology of $X$ is finitely generated over $R_G$.

Let $G$ be a cyclic group of order $n$, let $\bk = \F_p$ be a field of characteristic $p$ where $p$ is a prime dividing $n$. The ring $R_G$ is the tensor product of an exterior algebra $\Lambda$ and a polynomial ring $P_G$. Since $R_G$ is not necessarily an integral domain, we will also study the module structure of the equivariant cohomology over the subring $P_G$ from a slight different perspective of that in \cite{ptori} where this module structure is also discussed. One of the main results of this document is the following.

\begin{theorem}\label{teo:1.1}
	Let $G = \langle \sigma \rangle$ be a cyclic group of order $n$ and $X$ be a $G$-space. 
\begin{itemize}
\item Let $K \subseteq G$ be the subgroup of $G$ of order $p$. $X$ is $G$-equivariantly formal if and only if  $X$ is $K$-equivariantly formal and $X^K$ is $G$-equivariantly formal.
	
\item Consider the cohomological Serre spectral sequence arising from the fibration $X \rightarrow EG\times_G X \rightarrow BG$,	 and assume that it degenerates at the $E_2$-term. If the representation $(I+\sigma + \sigma^2 + \cdots + \sigma^{n-1}) = 0$ in $H^*(X)$, then $H^*(X)$ is a free $P_G$-module.
\end{itemize}
\end{theorem}
Notice that freeness over $P_G$ does not require $G$ to act trivially over the cohomology of the space.  When the action happens to be trivial, we  also characterize the $G$-equivariant cohomology of a space $X$ in terms the $K$-equivariant cohomology of the space as we state in the following result.
\begin{theorem}
  Suppose that $G$-acts trivially on the cohomology of $X$. There is an isomorphism of $R_K$-algebras
	\[ H^*_K(X) \cong H^*_G(X) \otimes_{R_G} R_K \]
\end{theorem}
This theorem is a consequence of $G$ and $K$ forming a flat extension pair. We also extend the ideas introduced in \cite{semi} from free extension pairs of groups to weak extension pairs and flat extension pairs of groups, where finite cyclic groups happen to fit.

With the theory developed in this document, we study the equivariant cohomology of some Riemann surfaces that arise from a polyhedral product with a canonical action of finite cyclic groups as started in \cite{Das} and also covered in a more general setting in \cite{poly-yu}. By considering $K_n$ the boundary of the $n$-gon as a finite complex, the canonical action of the cyclic group of order $n$ arises to an action on the polyhedral product $X = \mathcal{Z}_{K_n}(D^1,S^0)$. In particular, we characterize the values of $n$ and $p$ so $H^*_G(X;\F_p)$ is free over $R_G$ and $P_G$ as we summarize in the following result.

\begin{corollary}\ 
\begin{itemize}
 	\item $X$ is $G$-equivariantly formal over $\mathbb{F}_p$ if and only if $n = p = 2$ or $n = p = 3$.
	\item The $G$-equivariant cohomology of $X$ is free over $P_G$ if and only if $n=p=3$ or $n=4,p=2$.
	\item In any other case, the $G$-equivariant cohomology is not free over $P_G$.
\end{itemize}
\end{corollary}

This document is organized as follows: In section \ref{sec:cyclic} we review some properties of the cohomology of a classifying space of  a cyclic group, and we show that a group extension of cyclic groups induces a flat extension of cohomology rings. In Section \ref{sec:fep}, it is revisited the notion of free extension pairs in equivariant cohomology and extended to weak and flat extension pairs.  In Section \ref{se:mod} we study freeness of the equivariant cohomology over $R_G$ and $P_G$ and provide a proof of Theorem \ref{teo:1.1}. In Section 
\ref{se:Example} we provide an example to show that the freeness over the polynomial subring $P_G$ is not a hereditary property. Finally, in Section \ref{sec:poly}, the applications to actions on these polyhedral products are presented.

\textit{Acknowledgments.} The author wants to thank the late Fred Cohen for bringing into his attention the connection of equivariant cohomology with polyhedral products and, in particular, the work  of \cite{Das}. His valuable discussions, suggestions and comments on earlier versions of this document have led to considerable improvements in this work.
\section{Review on cohomology of cyclic groups}\label{sec:cyclic}

Let $p$ be a prime, $n > 1$ be any integer  and $C_n = \langle \sigma \rangle$ be a cyclic group of order $n$. If $p \mid n$, there is an isomorphism of algebras $H^*(BC_n; \F_p) \cong \F_p[x,y]/(x^2)$ where $x$ is of degree $1$ and $y$ is of degree $2$ when either $p > 2$ or $n > 2$ and $p= 2$. The remaining case is $H^*(BC_2. \F_2)  \cong F_2[t]$. Recall that the cohomology  $H^*(BC_n; \bk)$ is trivial over  a field $\bk$ of characteristic zero or $p \nmid n$. If $G = C_n$ and $K$  is  a subgroup of $G$, the following example will help to exhibit the structure of $R_K$ as $R_G$-module.

\begin{example}
	\normalfont
	Consider the short exact sequence $K = C_2 \xrightarrow{j} C_4 \xrightarrow{s} L = C_2$. With cohomology with coefficients over the field $\F_2$, there is a spectral sequence $E_2^{p,q} = H^p(BC_2, H^q(BC_2)) \Rightarrow H^{p+q}(BC_4)$ given as follows:

\begin{center}
\begin{figure}[h]
\centering
\includegraphics[scale=0.5]{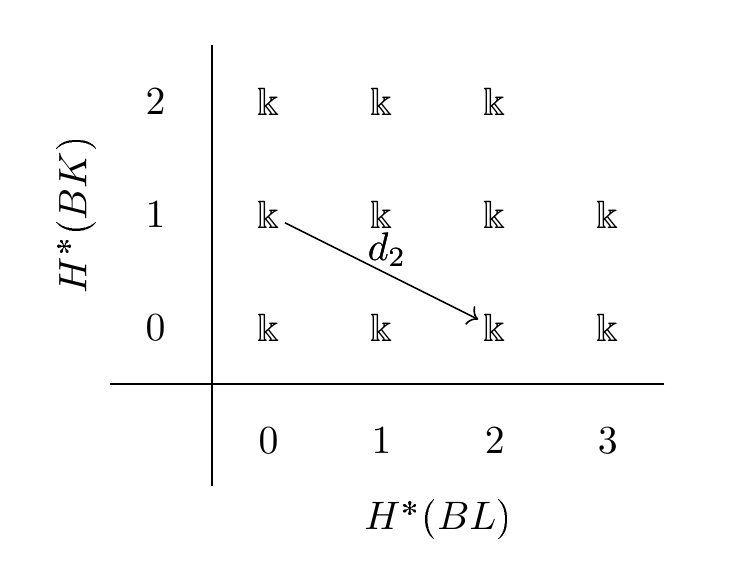}
\end{figure}
\end{center}

Choose isomorphisms $H^*(BL) \cong \F_2[t_L]$ and $H^*(BK) \cong F_2[t_K]$ where the generators are sitting in degree 1. We also have that $E_2 = \bk[t_L] \otimes \bk[t_K]$ and thus this page depends only on the value of $d_2(t_K)$. Since $H^k(C_4) = \bk$ for all $k \geq 0$, we must have that $d_2(t_K) = t_L^2$ and this gives the identities:
\begin{enumerate}
\item $d(t_K^2) = t_Kt_L^2 + t_L^2t_K = 0$.
\item $d(t_K^mt_L^n) = 0$ if $m$ is even, and it is non zero if $m$ is odd.
\end{enumerate}

Therefore, the spectral sequence degenerates at $E_3$ with $E_\infty$-page given as follows:

\begin{center}
\begin{figure}[h]
\centering
\includegraphics[scale=0.50]{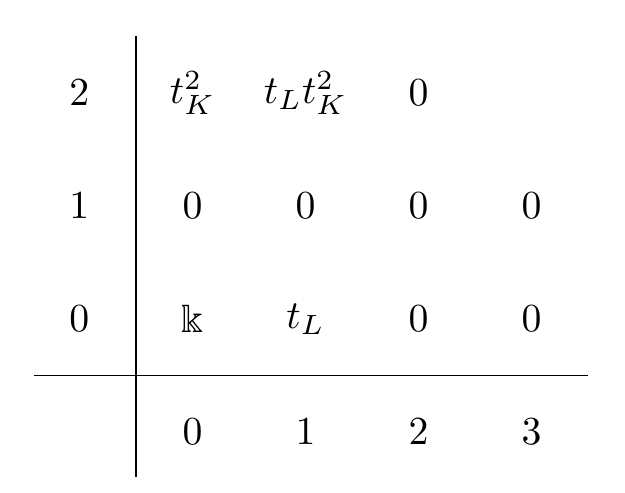}
\end{figure}
\end{center}

If we choose generators $H^*(C_4) \cong \bk[y] \otimes \Lambda[x]$, the above computation shows that $j^*(y) = t_K^2, j^*(x) =0$, and $s^*(t_L) = x$.
\end{example}
A similar spectral sequence argument holds for any cyclic group as we state in the following result.
\begin{proposition}\label{prop:2.1}
	Let $p$ be a prime. Consider a short exact sequence $K \xrightarrow{j} G \xrightarrow{s} L$ of cyclic groups of order divisible by $p$. The induced maps $j^*$ and $s^*$ in cohomology satisfy the following identities:
\end{proposition}
	\begin{enumerate}
	\item $\ker(j^*) = (H^1(BG;\F_p))$,  $\im(j^*) \cong H^{2*}(BK;\F_p)$.
	\item $\ker(s^*) = (H^{\geq 2}(BL;\F_p))$, $\im(s^*) \cong H^{\leq 1}(BG;\F_p)$. 
\end{enumerate}

\begin{corollary}\label{cor:2.1}
Let $G$ be a cyclic group of order $n$ and $\bk$ be a field of characteristic $p \mid n$. Let $R_G = H^*(BG; \bk)$ and $P_G = H^{2*}(BG;\bk)$. For any subgroup $ K \subseteq G$ such that $p \mid |K|$, there is an isomorphism of algebras $P_G \rightarrow P_K$ induced by the inclusion. Moreover, $R_K$ is a free $P_G$-module of rank 2.
\end{corollary}

Using these results, we see that $R_K$ is not free over $R_G$. However, in the next result we show that it is a flat ring extension.

\begin{proposition}\label{prop2.3}
Let $n$ be a positive integer and $p$ a prime dividing $n$. Let $G$ be a cyclic group of order $n$ and let $K$ be a subgroup such that $p \mid |K|$. Then $R_K$  is a flat as a module over $R_G$ when cohomology coefficients are considered over a field $\bk$ of characteristic $p$.
\end{proposition}

\begin{proof}
	We use the following characterization of flat modules: Let $M$ be a module over a ring $R$. $M$ is a flat module over $R$ if and only if for any finitely generated ideal $I \subseteq R$, the map $I \otimes_R M \rightarrow R \otimes_R M \cong M$ is injective. Using this fact, we will show that $R_K$ is flat over $R_G$. From Proposition \ref{prop:2.1}, we can choose isomorphisms $R_G \cong \bk[x_G,y_G]$ and $R_K \cong \bk[x_K,y_K]$ so that the $R_G$-module structure on $R_K$ is given by $x_G \cdot 1 = 0$ and $y_G \cdot 1 = y_K$. 
	
	Let $I = (r_1,\ldots,r_n)$ be a finitely generated ideal in $R_G$. Notice that each $r_i$ can be written uniquely in the form $r_i = x_Gp_i + q_i$ where $p_i,q_i \in P_G \cong \bk[y_G]$. Consider $J$ the ideal in $P_G$ generated by the $q_i's$. Since $R_K$ is a free $P_G$-module, it is a flat module and thus the map $J \otimes_{P_G} R_K \rightarrow R_K$ is injective. Now observe that the map $I \otimes_{R_G} R_K \rightarrow R_K$ fits in a commutative diagram
	\[
	\begin{tikzcd}
		J \otimes_{P_G} R_K  \rar\dar{\cong} & R_K \arrow[d,equal] \\
		I \otimes_{R_G} R_K  \rar & R_K
	\end{tikzcd}
	\]
	The map on the left is an isomorphism since for any $m \in R_K$ and $i=1,\ldots,n$, we have that $r_i \otimes m = q_i \otimes m$.
\end{proof}

\section{Free and flat extension pairs}\label{sec:fep}

In \cite{semi} it was studied the equivariant cohomology for pairs of groups $(G,K)$ such that the cohomology of $BK$ is free over the cohomology of $BG$ (over an underlying field $\bk$) as a generalization of the situation when $G$ is a compact connected Lie group and $K$ a maximal torus in $G$. The so called \textit{free extension pairs} allows us to relate the $G$-equivariant cohomology and $K$-equivariant cohomology for a $G$-space $X$. In this section we generalize the notion of free extension pairs to build up a theory that relates the equivariant cohomology for cyclic in  later sections.

\begin{definition}
	Let $K \subseteq G$ be topological groups. We say that $(G,K)$ is a weak extension pair if there is a group $K \subseteq G \subseteq \widetilde{G}$ such that $(\widetilde{G}, G)$ and $(\widetilde{G},K)$ are free extension pairs. In this case, we say that $\tilde{G}$ extends $(G,K)$.
\end{definition}
Note that immediately from this definition we get that a free extension pair $(G,K)$ is a weak extension pair (extended by $G$).   We now provide a known example of weak extension pairs that it is not a free extension pair.

\begin{example}\label{weakex}
\normalfont
Let $G$ be a cyclic group and $K$ be a subgroup. Then $(G,K)$ is a weak extension pair, extended by $S^1$, since $R_G \cong  H^*(BS_1)   \otimes \Lambda_G$   and $R_K \cong H^*(BS^1) \otimes \Lambda_K$ (over a field of characteristic $p \divides |K|)$ where $\Lambda_G$ and $\Lambda_K$ are exterior algebras on one generator of degree 1. Therefore, $(S^1,G)$ and $(S^1,K)$ are free extension pairs. As we remarked in the last section, $R_K$ is flat over $R_G$ but not free.
\end{example}

Similar to the discussion in \cite{semi} for free extension pairs. we prove the following result that relates equivariant cohomology for weak extension pairs when the action can be extended to an action of the group $\widetilde{G}$.
\begin{theorem}\label{teo:2.2}
	Let $(G,K)$ be a free $\widetilde{G}$-extension pair and $X$ be a $G$-space. Assume that the action of $G$  on $X$ can be extended to an action of $\widetilde{G}$. There is a natural isomorphism of $H^*(BK)$-algebras
	$$ H_K(X;\bk) \cong H^*_G(X;\bk) \otimes_{H^*(BG)} H^*(BK;\bk)$$
\begin{proof}
Since each $(\widetilde{G}, G)$ and $(\widetilde{G}. K)$ are free extension pairs, the result follows from the sequence of isomorphisms
	\begin{align*}
		H_K(X;\bk) & \cong H^*_{\widetilde{G}}(X;\bk) \otimes_{H^*(B\widetilde{G})} H^*(BK;\bk) \\
		& \cong H^*_{\widetilde{G}}(X;\bk) \otimes_{H^*(B\widetilde{G})} H^*(BG;\bk) \otimes_{H^*(BG)} H^*(BK;\bk) \\
		& \cong H^*_{G}(X;\bk) \otimes_{H^*(BG)} H^*(BK;\bk)
	\end{align*}
\end{proof}
\end{theorem}
Another notion that relates the cohomology of a group and a subgroup of it can be derived from free extension pairs is the following.
\begin{definition}
	Let $K \subseteq G$ be groups. We say that $(G,K)$ is a flat extension pair over $\bk$ if $R_K$ is a flat $R_G$-module.
\end{definition}

As an open question, it is yet to be determined what is the relation between weak extension pairs and flat extension pairs. Notice that as a consequence of Proposition \ref{prop2.3} and Example \ref{weakex}, the two notions are equivalent for cyclic groups. On the other hand, in a purely algebraic set up, the two notions are not comparable. For example, the inclusions $R = \bk[x^2,y^2] \rightarrow S = \bk[x^2, xy, x^2] \rightarrow T = \bk[x,y]$ make $T$ and $S$ free $R$-modules but $T$ is not flat over $S$.

\begin{theorem}
	Let  $(G,K)$ be a flat extension pair and $X$ be a $G$-space. Suppose that $G$-acts trivially on the cohomology of $X$. There is an isomorphism of $R_K$-algebras
	\[ H^*_K(X) \cong H^*_G(X) \otimes_{R_G} R_K \]
\end{theorem}
\begin{proof}
	Since the local coefficient system is trivial, the pullback diagram
	\[ 
	\begin{tikzcd}
		X_K\rar\dar & X_G\dar \\
		BK \rar & BG 
	\end{tikzcd},
	\]
	give rise to the Eilenberg-Moore spectral sequence $E_2 = \tor_{R_G}(H^*_G(X),R_K) \Rightarrow H^*_K(X)$. Since $R_K$ is flat over $R_G$, the spectral sequence is concentrated in the zeroth-column and thus $H^*_K(X) \cong H^*_G(X) \otimes_{R_G} R_K$.
\end{proof}

\section{The module structure over \texorpdfstring{$R_G$}{R\_G} and \texorpdfstring{$P_G$}{P\_G}}\label{se:mod}

Let $G = \langle \sigma \rangle $ be a cyclic group of order $n$ and $X$ be a $G$-space. The equivariant cohomology for elementary abelian $p$-groups over a field of characteristic $p$ has been studied by \cite{ptori} for $p$ any prime and in \cite{2tori}, \cite{tori-yu} for $p=2$ where some results carry over cyclic groups of any order and a field of positive characteristic.  

We will now study the module structure of the $G$-equivariant cohomology of a $G$-space $X$ as a module over $R_G$. As a consequence of Theorem \ref{teo:2.2} and Proposition \ref{prop2.3}, we have that if $G$-acts trivially on the cohomology of $X$ there is an isomorphism
\[ H^*_K(X) \cong H^*_G(X) \otimes_{R_G} R_K. \]
The result also holds if the action of $G$ on $X$ can be extended to a group $\widetilde{G}$ such that $(\widetilde{G}, G)$ and $(\widetilde{G},K)$ are free extension pairs. For example, $\widetilde{G} = S^1$.

\begin{example}
	\normalfont
	Let $X = S^1$, $G = C_4$ and $K = C_2$. Consider the action of $G = \langle\sigma\rangle$ on $X$ given by $\sigma\cdot x = \zeta_4^2x$ where $\zeta_4$ is a $4$th-root of unity. Notice that the action of $K = \langle \sigma^2 \rangle$ on $X$ is trivial and thus $H^*_K(X) \cong R_K \otimes H^*(X)$. 
	On the other hand, notice that this action is the restriction of the action of $\widetilde{G} = S^1$ on $X$ given by $g\cdot x = g^2z$. This action is transitive and we have that $H^*_{\widetilde{G}}(X) = R_K$. Since $(\widetilde{G}, G)$ and $(\widetilde{G},K)$ are free extension pairs, by Theorem \ref{teo:2.2} it follows that
	\[ H^*_G(X) \cong H^*_{\widetilde{G}}(X) \otimes_{R_{\widetilde{G}}} R_G \cong R_G \oplus R_G[1] \]
	and
	\[ H^*_K(X) \cong H^*_{\widetilde{G}}(X) \otimes_{R_{\widetilde{G}}} R_K \cong R_K \oplus R_K[1].\]
\end{example}

Let $b(X) = \sum_{k\geq 0} \dim_{\F_p} H^k(X;\F_p)$. We recall the following criterion for equivariant formality for cyclic groups of order prime as a consequence of the localization theorem.

\begin{theorem}
Let $p$ be a prime number and $K$ be a cyclic group of order $p$ , and  let $X$ be a $K$-space. Let $X^K$ denote the fixed point subspace of $X$. Then $X$ is $K$-equivariantly formal over $\F_p$ if and only if $b(X) = b(X^K)$.
\end{theorem}
Equivariant formality for cyclic groups of any order can be also characterized in terms of $p$-groups as we show in the following theorem as a particular case of the Localization theorem in equivariant cohomology \cite{Quil}

\begin{theorem}[Localization Theorem for Cyclic Groups]\label{teo:2.1}
	Let $G$ be a cyclic group of order $n$, $p$ be a prime $p\mid n$ and let $X$ be a $G$-space. Denote by $K$ the (unique) subgroup of $G$ of order $p$. Then the inclusion map $X^{K} \rightarrow X$ induces a map $H^*_G(X) \rightarrow H^*_G(X^{K})$ which is an isomorphism after localization under $S = P_G\setminus\{0\}$.
\end{theorem}
\begin{proof}
	For $x \in X$, let $G_x$ be the stabilizer subgroup of $x$, and let $j_x\colon G_x \rightarrow G$ be the inclusion. For a multiplicative subgroup $S \subseteq R_G$, let $X^S = \{ x \in X : j^*_x(s) \neq 0\, \forall s \in R_G \}$. The localization theorem in equivariant cohomology implies that the inclusion map $X^S \rightarrow X$ induces an isomorphism $S^{-1}H^*_G(X) \rightarrow S^{-1}H^*_G(X^S)$. 
	
	To prove our theorem, it is enough to show that $X^S = X^{K}$ when $S = P_G\setminus\{0\}$.
	
	By Proposition \ref{prop:2.1}, the map $j^*_x\restriction_S$ is injective if and only if $p \mid |G_x|$. Otherwise, the map is zero. 
	
	Consider the commutative diagram 
	\[
	\begin{tikzcd}
		R_G \rar{j^*_x}\dar{r} & R_{G_x}\dar \\
		R_{K} \rar{i_x^*} & R_{K_x}  
	\end{tikzcd}
\]
restricted to $S \subseteq P_G$.  Since the map $r\restriction_S$ is injective by Proposition \ref{prop:2.1}, we have that
\begin{align*}
x \in X^{K} &\Leftrightarrow K \subseteq G_x \\
&\Leftrightarrow p \mid |G_x| \\
&\Leftrightarrow j_x^*\restriction S\; \text{is injective}\\
&\Leftrightarrow x \in X^S
\end{align*}
\end{proof}
Using the cohomological spectral sequence arisen form the fibration $X \rightarrow X_G \rightarrow BG$ and an argument similar to the prime case \cite[Ch.III. prop. 4.16]{tomdieck}, we obtain the following consequence of the localization theorem \ref{teo:2.1}
\begin{corollary}\label{cor:3.7}
	Let $X$ be a $G$-space. Then  $X$ is $G$-equivariantly formal if and only if $X$ is $K$-equivariantly formal and $X^K$ is $G$-equivariantly formal.
\end{corollary}

From these results we can see that if $X$ is a $G$-space such that $X^G = X^K$, then $X$ is $G$-equivariantly formal if and only if $b(X^K) = b(X)$. However, there is no general characterization of the $G$-equivariant formality of $X$ in terms of its Betti numbers. The next example shows that even if $b(X) = b(X^K)$, the $G$-equivariant cohomology of $X$ fails to be free over $R_G$.
\begin{example}\label{ex: ex1}
\normalfont
Let $G = \langle \sigma \mid \sigma^4 = e \rangle$ and let $\sigma$ act on $X = S^1 \vee S^1$ as the reflection about the base point $p \in X$. The action has a single fixed point  the base point $p \in X$. Consider cohomology with coefficients in $\F_2$.
First notce that the subgroup $K = \langle \sigma^2 \rangle$ acts trivially on $X$ and thus $b(X) =b(X^K)$ and so $H_K(X) \cong H^*(BK) \otimes H^*(X)$. Now we use the equivariant Mayer-Vietoris exact sequence to compute $H^*_G(X)$. Let $U,V$ be the covering such that $U,V$ are $G$-contractibles to $p$, and $U\cap V$ is $G$-homotopic to $p \sqcup G/K$. We have then that $H^*_G(U) \cong H^*_G(V) \cong R_G$ and $H^*_G(U\cap V) \cong R_G \oplus R_K$. The restriction map $\rho\colon H^*_G(U) \rightarrow H^*_G(U\cap V)$ is given by $\alpha \mapsto (\alpha, r(\alpha))$ where $r\colon R_G \rightarrow R_K$ is induced by the inclusion of $K$ into $G$.
The Mayer-Vietoris long exact sequence induces a split exact sequence of $R_G$-modules
\[ 0 \rightarrow \coker(\rho + \rho)[1] \rightarrow H^*_G(X) \rightarrow \ker(\rho + \rho) \rightarrow 0 \]
As $\ker(\rho, \rho) \cong R_G$ and $\coker(\rho,\rho) \cong R_K$, we have that $H^*_G(X) \cong R_G \oplus R_K[1]$ as $R_G$-modules.  Since $R_K$ is not free over $R_G$, then $H^*_G(X)$ is not free over $R_G$ and so $X$ is not $G$-equivariantly formal.
\end{example}
\begin{remark}
\normalfont
In the previous example, if $R_K[1]$ is identified with $\F_2[\tilde{t}]u$ and $u$ is a generator of degree $1$ and a lifting of the generator $H^1(X)^G$, we have that the canonical map
\[ \Phi: H^*_G(X) \otimes_{R_G} R_K \rightarrow H^*_K(X)\]

satisfies that $\ker(\Phi) = (\tilde{t}u \otimes 1 - u \otimes t)$ and $\im(\Phi) = R_K \otimes H^*(X)^G$. This opens the following question:
{\textit{Let $X$ be a $G$-space such that $X$ is $K$-equivariantly formal. Is the image of the map canonical map of $R_G$-modules }}
\[  \Phi: H^*_G(X) \otimes_{R_G} R_K \rightarrow H^*_K(X) \cong R_K \otimes H^*(X)\]
$R_K \otimes H^*(X)^G$ ?
\end{remark}

Now we explore the module structure of $H^*_G(X;\F_p)$ over $P_G \cong \bk[y] \subseteq R_G$. As it was studied in \cite{ptori} for the prime case, the $G$-equivariant cohomology of a space might be free over $P_G$ and fails to be free over $R_G$, and the same does happen for the general cyclic group case.  Notice that the $G$-equivariant cohomology of the space from Example \ref{ex: ex1} is not free over $R_G$ , and  it is free as a module  over  $P_G$ by Corollary \ref{cor:2.1}. 

\medskip

To characterize the freeness over $P_G$, one can define an associated complex - \textit{The Atiyah-Bredon sequence}  $AB_G^*(X)$ -  to capture the freeness of $H^*_G(X)$ over $P_G$ \cite[Thm. 8.5]{ptori}. As an alternative result, we prove the following characterization that might be useful for computational purposes.

\begin{theorem}\label{teo3:2}
	Let $G = \langle \sigma \rangle$ be a cyclic group of order $n$ and $\bk$ be a field of characteristic $p \mid n$. Let $X$ be a $G$-space such that $(I + \sigma + \cdots + \sigma^{n-1})H^*(X) = 0$ and that $\ker(\sigma - I) \subseteq \im(\sigma - I)$. If the Serre spectral sequence of the Borel construction of $X$ degenerates at $E_2$, then $H^*_G(X)$ is a free $P_G$-module.
\end{theorem}

\begin{proof}
	Let $N = (I + \sigma + \cdots + \sigma^{n-1})$, for any $M$ be a $G$-module, the cyclic resolution of $\bk$ as trivial $G$-module induces an isomorphism
	\[ H^p(G;M) = \begin{cases}
		M^G &\text{if $p = 0$}\\
		M^G/NM &\text{if $p$ even}\\
		\ker(N)/(I-\sigma)M   &\text{if $p$ odd}  
	\end{cases}
\]

	Using this remark and that $(I + \sigma + \cdots + \sigma^{n-1})H^*(X) = 0$, the Serre spectral sequence has $E_2$-term given by $E_2^{p,*} = H^*(X)^G$ if $p$ is even and $E_2^{p,*} = H^*(X)/(I - \sigma)H^*(X)$.
	
	We now compute the spectral product $E_2^{p,0} \otimes E_2^{p',q} \xrightarrow{\cup} E_2^{p+p',q}$. In this page, this product coincides (up to a sign) with the topological cup product $H^p(BG; H^0(X)) \otimes H^{p'}(BG;H^q(X)) \rightarrow H^{p+p'}(BG; H^q(X))$. This map can  also be computed using a diagonal approximation $\Delta\colon P_* \rightarrow P_* \otimes P_*$ of the cyclic resolution $P_* \rightarrow \bk$ of $\bk$ as trivial $G$-module in the following way \cite[Ch.V.1]{brown} :
	
	$$ f \cup g \colon P_* \xrightarrow{\Delta} P_* \otimes P_* \xrightarrow{f \otimes g} H^*(X) \otimes H^*(X) \xrightarrow{\cup_X } H^*(X) $$
	is a representative in $H^*(BG; H^*(X))$. 
	
	Write $P_p = \bk[G]e_p$, let $f_p \in \hom_G(P_p,H^0(X))$ be a representative of the generator in $H^p(BG)$. Assume that $p'$ is even and let $g_q  \in \hom_G(P_{p'},H^q(X))$ be a representative of $u \in H^q(X)^G$. We have then that 
	$(f_p \cup g_q)(e_{p+p'}) = f_p(e_p)g_q(e_{p'}) = u \in H^1(X)^G$.  Therefore, the multiplication by the generator of $H^p(BG)$ is the identity if $p$ is even and it is zero if $p$ is odd. This shows that each  on row $q$, $E_2^{2*,q} \cong P_G \otimes H^{q}(X)^G$.

    Let now $v \in E_2^{p',q} = \ker(N)/\im(\sigma - I)H^q(X)$ for $p'$ odd. Choose $g_2 \in \hom_G(P_{p'}, H^q(X))$ a representative of $v$. A similar computation as before shows that the product $E_2^{p,0} \otimes E_2^{p',q} \xrightarrow{\cup} E_2^{p+p',q}$ is given by 
    
	\[
	(f_p \cup g_2)(e_{p+p'})  = \begin{cases}
		v & \text{$p$ even}\\
		\sum_{k=1}^{n-1} k\sigma^k(v) & \text{$p$ odd}
	\end{cases}
	\]
Notice that the element $\sum_{k=1}^{n-1} k\sigma^k(v) \in H^q(X)^G$ since $v \in \ker(N)$. Let $A = \sum_{k=1}^{n-1}k \sigma^k$ be the representation in $H^q(X)$. Notice that the polynomial $(x-1)$ divides $\sum_{k=1}^{n-1} kx^k$, and thus   the multiplication by the generator of $H^p(BG)$ is an isomorphism for any $p$.  Since each row in the $E_2$-term is a free $P_G$-module, and by assumption $E_2 \cong E_\infty$, we have that $H^*_G(X)$ is a free $P_G$-module.
\end{proof}

\section{Freeness over \texorpdfstring{$P_G$}{P\_G}) is not hereditary}\label{se:Example}

If $G$ is a topological group and $X$ is a $G$-space such that it is $G$-equivariantly formal, then for any subgroup $K \subseteq G$ it follows that $X$ is also $K$-equivariantly formal. This is a consequence of the factorization of the canonical map $H^*_K(X) \rightarrow H^*(X)$ through the surjective map $H^*_G(X) \rightarrow H^*(X)$. Let $G$ be a finite cyclic group. In this section we discuss an example of a $G$-space $X$  such that its $G$-equivariant cohomology is free over $P_G$ but its $K$-equivariant cohomology is not free over $P_K$ for some subgroup $K \subseteq G$. Therefore, freeness over $P_G$ is not hereditary over subgroups as freness over $R_G$ is.

Let $G = C_{n}$ and $\bk = F_p$ be a field of characteristic $p \divides n$ and $K = C_{p}$. Let $X = \Sigma G$ be the suspension of the discrete group $G$. The left multiplication of $G$ on itself induces an action of $G$ on $X$. It can be seen that $X \simeq \bigvee_{n-1} S^1$, $X^G = X^K = \{p_1, p_2\}$, and that $G$ does not act trivially on $H^1(X;\bk)$. This shows that  $b(X^K) = 2 < b(X) = n$ and thus $X$ is not $G$-equivariantly formal for $n > 2$ (and neither $K$-equivariantly formal). The $G$-equivariant cohomology was computed in \cite[Sec.10]{ptori} and it is shown that
 
\[ H^*_G(X) \cong  (R_G \oplus R_G)/(\bk \oplus \bk) \] as $R_G$-modules. Unless $n \neq 2$, $H^*_G(X)$ is not a free $R_G$-module, but it is free over $P_G$.

On the other hand, we will use Theorem \ref{teo3:2} to see that $H^*_G(X)$ is free over $P_G$ and that $H^*_K(X)$ contains $P_K$-torsion if $n \neq p$. 
Firstly, choose a basis $\{e_1,\ldots, e_{n-1}\}$ of $H^1(X;\F_p)$ such that $C_n$ is generated by the representation matrix 
\[ A = \left( 
\begin{array}{c|c}
	\bm{0} & -1  \\ \hline
	I_{n-2} & \bm{-1} 
\end{array}
\right)
\]
on $H^1(X:\F_p)$. It can be checked that  the map  $N = I + A + A^2 + \cdots + A^{n-1}$ is zero since $Ne_i = \sum_{k=1}^{n-1}e_k - \bm{1} = 0$ for any $i=1,\ldots,n-1$.  Furthermore,  $\ker(A-I)$ is spanned by the vector $[1,2,....,n-2,-1]^T \in \F_p^{n-1}$. Therefore, the $E_2$-term of the Serre spectral sequence arising from the fibration $X \rightarrow X_G \rightarrow BG$ is given by
 
\begin{center}
	\begin{sseqdata}[name =tikz background example2,cohomological Serre grading,classes = {draw = none},grid =none,x label ={ $H^*(G)$ },y label ={ $H^*(X)^G$ }]
		\class["\bk"](0,0)
		\class["\bk"](2,0)
		\class["\bk"](0,1)\d["d_2"{ pos =0.7, yshift =1.8em}]2(0,1)
		\class["\bk"](1,0)
		\class["\bk"](1,1)
		\class["\bk"](3,0)
		\class["\bk"](3,1)
		\class["\bk"](2,1)
		\class["0"](0,2)
		\class["0"](1,2)
		\class["0"](2,2)
	\end{sseqdata}
	\printpage[name =tikz background example2,page =2]
\end{center}
where $H^{2*}(BG;H^1(X)) = H^1(X)^G \cong \F_p$ and $H^{2*+1}(BG;H^1(X)) = \coker(A-I) \cong \bk$. 

The localization theorem for cyclic groups implies that $H^k_G(X; \bk) \cong H^k_G(X^K; \bk)$ for $k > 1$, and that the map $H^1_G(X) \rightarrow H^1_G(X^K)$ is surjective.  Therefore, by degree reasons, we must have that the spectral sequence degenerates and so $E_2 \cong E_\infty$. Finally, by Theorem \ref{teo3:2}, we have that $H^*_G(X)$ is a free $P_G$-module, and it is of rank $4$.

Now we will show that the $K$-equivariant cohomology of $X$ is not free over $P_K$. First, comparing the spectral sequences $E_r$ and $(\tilde{E}_r)$ for both $X_G$ and $X_K$ respectively, we see that $\tilde{E}_r$ also degenerates at the second page. The representation of $K$ on $H^1(X: \F_p)$ is given by the matrix $B = A^{n/p}$. Notice $\tilde{N} = I + B + \cdots B^{p-1} = 0$ if and only if $ p = n$ as for any $i =1,\ldots,n-1$ and $k=0,\ldots, p-1$, it holds that $B^ke_i = e_t$ or $B^ke_i = \bm{-1}$ where $t$ is uniquely determined by $i$ and $k$. This shows that $\im(\tilde{N}) \neq 0$ and thus the $\tilde{E}_2$-term contains $P_K$-torsion.

\section{Actions on Some Polyhedral Products}\label{sec:poly}

Polyhedral products are spaces that arise as cartesian products based on a simplicial complex. They extend the notion of moment angle complex in toric topology and they have recently been  objects of interested in combinatorics, symplectic geometry and homotopy theory among other fields. A compendium of the main properties and recent development on this class of spaces is found in \cite{tbari} and the references therein. We start this section by briefly introducing polyhedral products and setting up the preferred notation to be used throughout this document.
  
For any integer $n$, let $[n] = \{1,2,\ldots,n\}$. A finite simplicial complex $K$ over $[n]$ is a subset $K \subset \mathcal{P}([n])$ such that $\{i\} \in K$ for all $i=1,\ldots,n$ and if $\tau \subseteq \sigma$ and $\sigma \in K$, then $\tau \in K$. Each $\sigma \in K$ is called a face, and a simplicial complex can be described by its maximal faces. For example, the boundary of the $n$-gon is the simplicial complex whose maximal faces are of the form $\{i,i+1\}$ and $\{1,n\}$ for $i=1,\ldots,n-1$.

Let $(X,A)$ be a topological pair and $K$ be a a simplicial complex. For any $\sigma \in K$ we denote by $(X,A)^\sigma$ the space 
$Y_1 \times \cdots \times Y_n \subseteq X^n$ where $Y_i = X$ if $i \in \sigma$ and $Y_i = A$ if $i \notin \sigma$.

Let $n \geq 3$ be an integer and $K_n$ be the boundary of the $n$-gon. Denote by $X = Z_{K_n}(D^1,S^0)$ the polyhedral product over $K_n$. $X$ is homeomorphic to a Riemann surface of genus $1+(n-4)2^{n-3}$ \cite{coexter}. Let $G$ be the cyclic group of order $n$ which acts canonically on $X$ by identifying it with a subgroup of $Aut(K_n) = D_{2n}$ the dihedral group of order $2n$. In \cite[Ch.IV,V]{Das} it was described the $\mathbb{Z}[G]$-module structure of $H^*(X;\mathbb{Z})$ and showed that the integral cohomological Serre spectral sequence associated to the fibration $X \rightarrow X_G \rightarrow BG$ degenerates at the $E_2$-term. It was also shown that it collapses when coefficients over a field of characteristic zero or coprime to $n$ are considered.   We  will use the techniques described in previous sections to study the $G$-equivariant cohomology of this polyhedral product when the characteristic of the ground field in cohomology divides $n$.

Let $p$ be a prime dividing $n$, $K \subseteq G$ be the subgroup of order $p$ and $\mathbb{\bk}$ be a field of characteristic $p$. 
We start by studying the equivariant formality of $X$ under the action of $K$. The first result is the following.

\begin{proposition}\label{prop5:1}
$X$ is $K$-equivariantly formal over $\mathbb{F}_p$ if and only if $n = p =3$ or $n = 4$, $p=2$.
\end{proposition}
\begin{proof}
Recall that $X$ is $K$-equivariantly formal if and only if $b(X) = b(X^K)$. The elements of $X$ are represented by tuples $(x_1,\ldots,x_n) \in (S^0)^n$ such that two consecutive elements $x_i,x_{i+1} \in D^1$ or $x_1,x_n \in D^1$.  Let $\sigma$ be the generator of $G$ that acts on $X$ as the cyclic permutation and thus the generator of $K$ is $\sigma^{n/d}$. This shows that $(x_1,\ldots,x_n) \in X^K$ if and only if for all $i$, $x_i \in S^0$ and $x_i = x_{i+\frac{n}{p}}$ (the sum in the subindex is considered mod $p$). This shows that $X^K \cong (S^0)^{\frac{n}{p}}$ and thus $b(X^K) = 2^{n/p}$. Therefore, $X$ is equivariantly formal if and only if 

\[ 2^{n/p} = 2 + 2(1+(n-4)2^{n-3}).\]

This equation is equivalent to $2^{n/p -2} = 1 + (n-4)2^{n-4}$. Notice that if $n = p$, then the left hand side is equal to $\frac{1}{2}$\.  The right hand side is equal to $\frac{1}{2}$ if and only if $n = 3$. Notice that the case $n>p>2$ is impossible by taking the equation modulo $2$. Therefore, the only remaining valid solution is $n = 4$ and $p = 2$.
\end{proof}

The following results of this section, require a combinatorial argument that involves \emph{Lyndon words}. A Lyndon word of length $n$ is a word in the alphabet $\{0.1\}$ which is minimal (with respect to the lexicographic order) among all its circular rotations.  We denote by $\ell_n$ the number of Lyndon words of length $n$. If $\omega$ is such a word, the number of times that the subword `$01$' appears on $\omega$ is denoted by $\iota(\omega)$.

\begin{theorem}
	Let $\bk$ be a field of characteristic $p \mid n$. The cohomological Serre spectral sequence associated to the fibration $X \rightarrow X_G \rightarrow BG$ degenerates at the $E_2$-term.
\end{theorem}
\begin{proof}
As in the proof of \ref{prop5:1}, the fixed point $X^K \cong (S^0)^{n/p}$ and the action of $G$ restricted to $X^K$ coincides with the cyclic permutation of $G/K = C_{n/p}$ on $X^K$. Since $X^K$ is totally disconnected, $H^*_G(X^K;\bk)$ is the direct sum of the $G$-equivariant cohomologies of the $G$-orbits of $X$. The non-trivial orbits are represented by the Lyndon word of length $d$ where  $d \mid \frac{n}{p}$. This implies that there is an isomorphism of $R_G$-modules
\[ H^*_G(X^K;\bk) \cong \bigoplus_{w \in X^K/G} H^*(BG_w;\bk)\]
where $G_w \subseteq G$ is the stabilizer subgroup of the representative $w \in X^K/G$. Notice that $K \subseteq G_w$ and therefore $\dim H^k(BG_w;\bk) = 1$ for all $k$. If $\ell_m$ denotes the number of Lyndon words of length $m$, it can be seen that
\[ \dim H^k_G(X^K;\bk) = \sum_{d \mid \frac{n}{p} } \ell_{d}.\]
By the Localization theorem, we have that $\dim H^k_G(X;\bk)  = \dim H^k_G(X^K;\bk)$ for $k > 2$. To finish the proof, it is enough to show that in the $E_2$-term of the Serre spectral sequence associated to the fibration $X \rightarrow X_G \rightarrow BG$, we must have 
\begin{equation}\label{eq5.1}
	\sum_{i=0}^\infty \dim E_2^{i, k-i} = \sum_{d \mid \frac{n}{p} } \ell_{d}.
\end{equation}
Observe that the spectral sequence contains only three non-zero rows as $H^q(X) = 0$ for $q > 2$. The local coefficient system is trivial when $b = 0,2$, and the $G$-module structure of $H^1(X)$ was described in \cite[Sec. 4.3]{Das}. Indeed, if $D_{n,p} = \{ d \in \mathbb{Z}_{+} : d \mid n, p \mid \frac{n}{d} \}$, it is shown that for $k > 0$,  $\displaystyle\dim H^k(BG;H^1(X)) = \sum_{ d \in  D_{n,p}, d \neq 1 } \ell_d $. This shows that $\displaystyle\dim E_2^{i, k-i} = \sum_{ d \in  D_{n,p} } \ell_d$ for $k > 1$. Finally, notice that $d \in D_{n,p}$ if and only if $d \mid \frac{n}{p}$ which implies that the equality in \ref{eq5.1} holds.
\end{proof}

Let $L(n,k)$ be the number of $n$-Lyndon words with exactly $k$ blocks of zeros and $\mathcal{L}_n^+$ be the set of Lyndon words that have $k$ blocks of zeros with $k > 1$.

\begin{corollary}
	The $G$-equivariant cohomology of the polyhedral product $Z_{K_n}(D^1,S^0)$ is free over $P_G$ if and only if $n=p=3$ or $n=4,p=2$.
\end{corollary}

Let $X = Z_{K_n}(D^1,S^0)$. Following \cite[Thm.4.3.3]{Das}, there is an isomorphism of $\bk[G]$-modules
\[ H^1(X;\bk) \cong  \bigoplus_{\substack{ { w \in \mathcal{L}_d} \\ {d \in D_{n,p}} }} M_w \oplus \bigoplus_{w \in \mathcal{L}_n^+} \bk[G]^{\iota(w) -1} \]
where the representation $A_w$ of $G$ in $M_w$ satisfies that $N_w = I + A + \cdots + A^{n-1} = 0$ and thus they generate the $P_G$-free part of $H^*_G(X)$. On the other hand, the $P_G$-torsion part of $H^*_G(X)$ is concentrated in degree $1$ and it has dimension $|\mathcal{L}_n^+| = l_n  - L(n,1)$. Finally, we see that the $G$-equivariant cohomology of $X$ contains $P_G$-torsion if and only if $l_n = L(n,1)$ which only holds if and only if $n=3$ or $n=4$.

\end{document}